\documentclass[11pt,oneside]{article}
\usepackage[margin=1in]{geometry}

\usepackage{amsmath,amssymb,amsthm,amscd}
\usepackage{calrsfs,enumitem}
\usepackage[dvipsnames]{xcolor}
\usepackage{mathtools}
\usepackage{graphicx}
\usepackage{euscript}
\usepackage{amsfonts}
\usepackage[abs]{overpic}	

\usepackage{epsfig}
\usepackage{epstopdf}
\usepackage{eufrak}
\usepackage{tikz,underscore}
\usepackage{tikz-cd}
\usepackage{bm}
\usepackage{float}

\usepackage{xcolor}

\definecolor{lred}{RGB}{226,104,104}
\definecolor{lgrey}{RGB}{157,187,215}
\definecolor{lpurple}{RGB}{176,156,255}

\newcommand{\Addresses}{{
  \bigskip
  \footnotesize

  Beibei~Liu, \textsc{Department of Mathematics, The Ohio State University}\par\nopagebreak 
  \textsc{100 Math Tower, 231 W 18th Ave, Columbus, OH 43201}\par\nopagebreak
  \textit{E-mail address}: \texttt{liu.11302@osu.edu}

  \medskip

  Lisa~Piccirillo, \textsc{Department of Mathematics, MIT and The University of Texas at Austin}\par\nopagebreak 
  \textsc{PMA 8.100, 2515 Speedway, Austin, TX 78712}\par\nopagebreak
  \textit{E-mail address}: \texttt{piccirli@mit.edu}

}}

\usepackage[linktocpage]{hyperref}
\hypersetup{
    colorlinks=false,
    linkcolor=red,
    hypertexnames=false,
    filecolor=black,      
    urlcolor=black,
    citecolor=blue
}
\numberwithin{equation}{section}

\theoremstyle{plain}
\newtheorem{theorem}{Theorem}[section]

\newtheorem{lemma}[theorem]{Lemma}

\newtheorem{corollary}[theorem]{Corollary}

\theoremstyle{definition}

\newtheorem{definition}[theorem]{Definition}

\def\s{\mathfrak{s}}

\title{Bounding the Dehn surgery number by 10/8}
\author{Beibei Liu \and Lisa Piccirillo}

\begin{document}
\maketitle
\begin{abstract}
We provide new examples of 3-manifolds with weight one fundamental group and the same integral homology as the lens space $L(2k,1)$ 
which are not surgery on any knot in the three sphere. Our argument uses Furuta's 10/8-theorem, and is simple and combinatorial to apply.

\end{abstract}


\section{Introduction}

Every closed, oriented 3-manifold is obtained by a Dehn surgery on a link in the three-sphere \cite{Lic, Wall}. This suggests a natural way to quantify the complexity of a 3-manifold; the \textit{Dehn surgery number} of a 3-manifold $M$ is the minimal number of components of a link in $S^3$ which admits a Dehn surgery to $M$. This is a deceptively simple invariant; other than a handful of classical lower bounds coming from the fundamental group or homology or the number of reducing spheres, it is notoriously difficult to give lower bounds. In what follows, we will assume all 3-manifolds are irreducible, closed, oriented, and have the integral homology of some lens space $L(r,1)$ (including $r\in \{0,1\}$). To date there are no examples in the literature of such 3-manifolds with surgery number greater than two.\footnote{Daemi-Miller-Eismeier announced examples in 2022, their preprint is forthcoming.}
Even showing that there are 3-manifolds with Dehn surgery number at least two is a question with a rich history.

In 1990, Boyer and Lines \cite{BL} used the Casson invariant to construct homology lens spaces with Dehn surgery number at least two. 
In 1996 Auckly \cite{Auck} gave homology spheres with Dehn surgery number at least two; his proof relies on Taubes' end-periodic diagonalization theorem \cite{Taube} and he gives both toroidal and hyperbolic examples. In 2016, Hom, Karakurt and Lidman \cite{HKL} constructed Seifert fibered manifolds with Dehn surgery number two using Heegaard Floer homology, and Hom and Lidman later gave more hyperbolic examples \cite{HL}. 
In 2019, Hedden, Kim, Mark, and Park \cite{HKMP} constructed homology $S^1\times S^2$s with Dehn surgery number at least two; their examples can be taken to be either Seifert fibred or hyperbolic and their obstruction uses the Rohlin invariant or Heegaard Floer homology respectively. 
In 2022 Sivek and Zentner used the SU(2) character variety of fundamental groups to provide homology lens spaces with Dehn surgery number 2 \cite{sivekzentner}.
%

In this note, we give a new argument for showing that the Dehn surgery number is at least two using Furuta's 10/8-theorem \cite{Furuta}. We do not demonstrate any fundamentally new topological phenomena. The contribution is our argument,  which is simple and combinatorial to apply.  Applications include re-proofs that for all $k\in\mathbb{Z}$ there are 3-manifolds with the homology of $L(2k,1)$ which are not surgery on any knot in the three-sphere. Let $L_n$ be the 2-component link shown in Figure \ref{fig:Ln}; here $n$ is an odd integer and both components $K_1, K_2$ are torus knot $T(n, n+1)$. 

\begin{theorem}
\label{thm:evenknot}
For any odd integers $p,q$ and odd integer $n$ sufficiently large, the surgery manifold $S^3_{p,q}(L_n)$ is not a surgery on a knot in $S^3$. 
\end{theorem}

The proof gives explicit lower bounds on $n$. For $n\ge 2$, $S^3_{p,q}(L_n)$ is an irreducible graph manifold with weight 1 fundamental group. 
We note that our examples are similar in spirit to those in \cite{HKMP, sivekzentner}, which are also splices of torus knot complements. 

One can choose odd integers $p$ and $q$ such that any even integer $2k$ arises as the order of $H_1(S^3_{p, q}(L_n))$, for example by choosing $p=2k+1, q=1$.  This yields the following corollary.

\begin{corollary}
\label{coro:even}
For any integer $k$ and sufficently large odd integer $n$, the homology lens space $L(2k,1)$ given by $S^{3}_{2k+1, 1}(L_n)$ is not a surgery on a knot in $S^3$.
\end{corollary}  


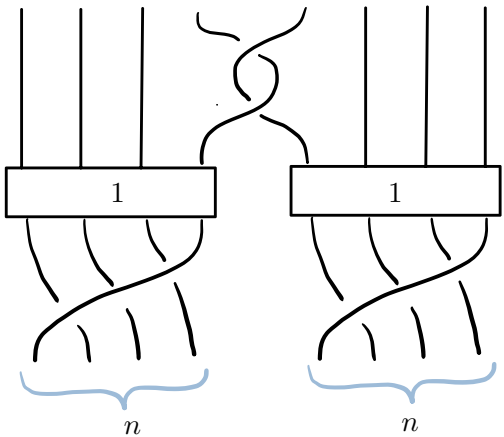
\begin{figure}[H]
\centering
\begin{overpic}[width=.4\textwidth, tics=20]
  {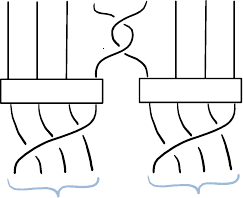}
   \put(45, -9){$n$}
      \put(150, -8){$n$}
   \put(145, 79){$1$}
      \put(40, 79){$1$}
  \end{overpic}
\caption{The link $L_n$ is obtained by braid closing the top and bottom of this figure.}
\label{fig:Ln}
\end{figure}



The arguments in this paper are somewhat robust under connected sums, for example 
one can extend  Corollary \ref{coro:even} to show that for any odd integers $p, q$, the connected sum $\#^m S^3_{p, q}(L_n)$ is not a surgery on any $m$-component link in the three-sphere for sufficiently large odd $n$. 






\section*{Acknowledgements} The first author is partially supported by the NSF grant DMS-2203237. The second author is supported in part by a Sloan Fellowship, a Clay Fellowship, and the Simons collaboration ``New structures in low-dimensional topology''.
We are grateful to Anthony Conway, Jen Hom, and Tye Lidman for helpful conversations. 

\section{Preliminaries}

Our argument for bounding the Dehn surgery number of a 3-manifold $M$ will require obstructing the existence of certain spin fillings of $M$.  See \cite[Section 1.4.2]{GomS} for discussion of spin structures.   To obstruct $M$ from having a spin filling $X$ with a certain Euler characteristic and signature, we can assume such an  $X$ exists and glue it to some known spin filling $X'$ of $-M$, then look for a contradiction with Furuta's $10/8$-theorem \cite{Furuta} applied to the resulting closed spin 4-manifold.

\begin{theorem}\cite[Theorem 1]{Furuta}
\label{thm:10/8}
Let $X$ be a closed, spin, smooth 4-manifold with $b_2(X)\neq 0$ and indefinite intersection form. Then 
$$4b_2(X)\geq 5|\sigma(X)|+8.$$
where $b_2(X)$ is the second Betti number of $X$ and $\sigma(X)$ is the signature of $X$. 
\end{theorem}

Obstructing spin fillings of $M$ in this way is a classical argument,  and we were particularly inspired here by the work of Donald-Vafaee \cite{Dvafa}. For an argument like this to be practical, one needs to be able to construct spin fillings to use as $X'$, i.e. honest spin fillings of $-M$ with various Euler characteristics and signatures.  Towards this,  we recall the so-called characteristic sublink calculus. 

We will focus on extending spin structures on 3-manifolds over compact 4-manifolds built out of 0- and 2-handles. Call such 4-manifolds \textit{2-handlebodies.} Suppose that $M$ is the boundary of a 2-handlebody $X$ which is described by attaching 2-handles to a framed link $L\subset S^3$. Such 2-handlebody $X$ is also called the \emph{trace} of the framed link.   Define characteristic sublinks of $L$ as follows:

\begin{definition}
Given a framed link $L=K_1\cup K_2\cup\cdots\cup K_n$ in $S^3$, a characteristic sublink $L'\subset L$ is a sublink such that for each $K_i$ in $L$, the framing $lk(K_i, K_i)$ is congruent mod $2$ to $lk(L', K_i)$.
\end{definition}

By work of Kaplan \cite{Kaplan}, there is a bijection between the spin structures on the 3-manifold $M$ and characteristic sublinks of $L$, see the discussion in \cite[Section 5.7]{GomS}, in particular \cite[Proposition 5.7.11]{GomS}.
This perspective on spin structures is particularly convenient for us because of the following: if a spin 3-manifold $(M, \s)$, described as integral surgery on a link in $S^3$, has that the empty sublink is characteristic then the 2-handlebody filling defined by the same framed link gives a spin filling of $(M,\s)$, see \cite[Theorem 5.7.14]{GomS}. 

Thus, if one is interested in finding a spin filling $X'$ of some spin structure $\mathfrak{s}$ on $M$, one can apply the following Kirby moves to modify a surgery diagram for $M$ until the empty link is a characteristic sublink. For reference, see \cite[Section 5.7]{GomS}, in particular, the discussion after Figure 5.47 in \cite{GomS}.
Notice that these moves change the characteristic sublink $L'$ and the diffeomorphism type of $X$, but preserve the diffeomorphism type of $M$. 

Given an integral surgery diagram for $M$ and some initial characteristic sublink $L'$, one can
\begin{enumerate} 
\item Add or remove an $\pm 1$-framed split unknot to $L'$ to get a new characteristic sublink $L''$. This corresponds to blow-up/downs on $X$. 
\item Given two knots $K_1, K_2 \subset L'$ one can replace $L'$ by a new characteristic sublink $L''=(L'\setminus (K_1\cup K_2))\cup (K_1\natural K_2)$, where $K_1\natural K_2$ denotes any band sum of $K_1$ and $K_2$. 
This corresponds to handle slides on $K_1$ and $K_2$, and preserves the diffeomorphism type of $X$.
\item Given knots $K_1 \subset L'$ and $K_2\subset L\setminus L'$ one can replace $L'$ by a new characteristic sublink $L''=(L'\setminus K_1)\cup ((K_1\natural K_2)\cup K_2)$, where $K_1\natural K_2$ denotes any band sum of $K_1$ and $K_2$. 
This also corresponds to handle slides on $K_1$ and $K_2$, and preserves the diffeomorphism type of $X$.
\end{enumerate}

\begin{figure}[H]
\centering
\begin{overpic}[width=.9\textwidth, tics=20]
  {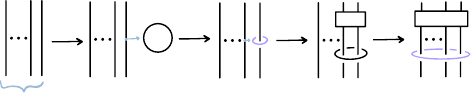}
   \put(20, -7){\color{lgrey}{$n$}}
      \put(40, 70){$f$}
            \put(119,70){$f$}
                        \put(238,70){$f-1$}
   \put(151, 34){$-1$}
      \put(238, 34){$\color{lpurple}{-1}$}
                              \put(331,70){$f-4$}
   \put(309, 62){$-1$}
      \put(328, 22){$-1$}
                                    \put(423,70){$f-n^2$}
   \put(389, 63){$-1$}
      \put(420, 22){$\color{lpurple}{-1}$}
  \end{overpic}
\caption{A convenient move in the characteristic sublink calculus.  Here $n$ must be odd. }
\label{fig:move}
\end{figure}
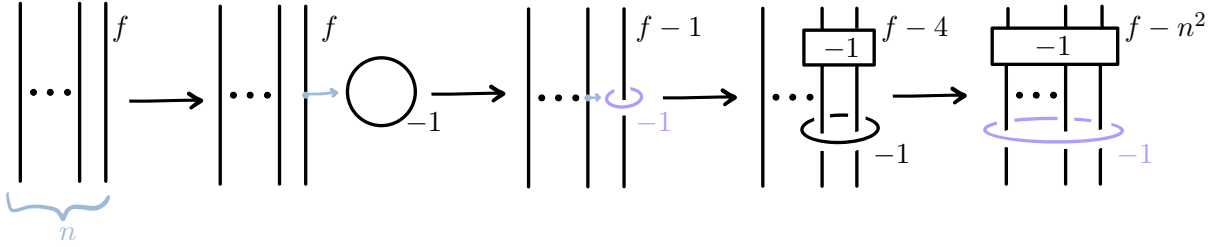

The characteristic sublink calculus sequence illustrated in Figure \ref{fig:move} for blowing up across an odd number of strands of $L'$ is a particularly convenient combination of these three moves. Handle slides are marked in grey. The black components are a part of the characteristic sublink and the purple components are not a part of the characteristic sublink.

\subsection{Some spin fillings of certain lens spaces}
For some cases of the proof of Theorem \ref{thm:evenknot}, we will require the following lemma about even surgery descriptions of lens spaces. This lemma is surely known to experts. 
\begin{lemma}\label{lem:evenchain}
For any nonzero even integer $s$ and integer $t$ coprime to $s$, the lens space $L(t, s)$ can be obtained by surgery along a linear chain of unknots $U_1, \cdots, U_k$ with even surgery coefficients and with $k\leq |s|$. 
\end{lemma}

\begin{corollary}\label{coro:lens}
For any nonzero even integer $s$ and integer $t$ coprime to $s$, the lens space $L(t, s)$ bounds a spin 4-manifold with $b_2$ and $|\sigma|$ bounded above by $|s|$.
\end{corollary}
\begin{proof}
With the surgery presentation in Lemma \ref{lem:evenchain}, the empty link is characteristic. As such, the trace of the framed link from the surgery description of $L(t,s)$ in Lemma \ref{lem:evenchain} is a spin filling of $L(t,s)$. 
\end{proof}

\begin{proof}[Proof of Lemma \ref{lem:evenchain}]
The proof will be constructive. Begin by writing 
$$\dfrac{t}{s}=m+\dfrac{l}{s}$$
where $m$ is an even integer, $\gcd(l, s)=1$ (so in particular, $l$ is odd), and $0<|l|<|s|$. Then $L(t, s)$ can be regarded as a surgery on the Hopf link with surgery coefficients $m$ and $-\dfrac{s}{l}$, see for example \cite[Section 5.3]{GomS}. If $l=\pm 1$, $L(t,s)$ is a surgery on the Hopf link with even surgery coefficients $m, \mp s$, and we are done.   

Otherwise, we write
$$-\dfrac{s}{l}=m_1+\dfrac{l_1}{l}$$
where $m_1$ is even, $\gcd(l_1, l)=1$ and $0<|l_1|<|l|$. Note $l_1$ must be even since $s$ is even.
Now we can describe $L(t, s)$ as surgery on a linear chain of unknots $U_1, U_2, U_3$ with surgery coefficients $m, m_1, -\dfrac{l}{l_1}$. Note that since $l_1$ is even, this chain is not integrally framed. As before, we can write
$$-\dfrac{l}{l_1}=m_2+\dfrac{l_2}{l_1}$$
where $m_2$ is even, $\gcd(l_1, l_2)=1$ and $0<|l_2|<|l_1|$. Since $l_1$ is even, $l_2$ is odd. If $l_2=\pm 1$, $L(t, s)$ is the surgery  on a linear chain of unknots $U_1, U_2, U_3, U_4$ with surgery coefficients $m, m_1, m_2, \mp l_1$  which are all even, and we are done. 

If $|l_s|>1$, we repeat the steps in the preceding paragraph. Note that $m_i$ is always chosen to be even, so every iteration yields a description of $L(t,s)$ as surgery on a linear chain of unknots where all but (perhaps) the final component is an even integer. This process yields a sequence of integers
$$|s|>|l|>|l_1|>|l_2|>\cdots>|l_{2r-1}|>|l_{2r}|>0$$ 
with $l_{2i-1}$ even and $l_{2i}$ odd for all $1\leq i\leq r$. The process terminates when we reach an $r$ such that $|l_{2r}|=1$; this must occur for some $r$ satisfying $2(r+1)\leq |s|.$


Thus we have described the lens space $L(t, s)$ is a surgery along a linear  chain of unknots $$U_1, U_2, \cdots, U_{2r+1}, U_{2r+2}$$ with even surgery coefficients $m, m_1, m_2, \cdots, m_{2r}, \pm l_{2r-1}$ and where $2r+2\le |s|$.  
\end{proof}

\section{Proofs of main results}
\label{sec:proofs}

\noindent
We begin by showing that Dehn surgery number of certain integer homology $S^1\times S^2$s is $2$; this setting is more straightforward than for other homology lens spaces since any description of such a manifold as surgery on a knot in $S^3$ must have surgery coefficient $0$. (Otherwise the surgery coefficient can be any $s/t\in\mathbb{Q}$ so long as $|s|=|H_1(M)|$.) 

\vspace{3mm}
\noindent
\textbf{Proof of Theorem \ref{thm:evenknot} in the case k=0:} We begin with an outline of the argument. To have $k=0$ we must have $p=q=\pm 1$. We will prove the $p=q=1$ case, and  the other is identical. The Dehn surgery number of $S^{3}_{1,1}(L_n)$ is bounded above by 2 by definition.  If $S^{3}_{1,1}(L_n)$ were surgery on some knot $K$, then it would be the boundary of the 0-trace $X_0(K)$ of that knot; note that the 0-trace is spin and has $b_2=1$ and $\sigma=0$. We will show that for both spin structures on $S^{3}_{1,1}(L_n)$ there is a spin 2-handlebody filling $X'_i$ with very large signature rel. second Betti number. Gluing one of $X_i'$ fillings to the putative 0-trace we will get a closed spin manifold which contradicts Theorem \ref{thm:10/8}.

Let us start by building these $X'_i$. Observe that $S^{3}_{1,1}(L_n)$ has two spin structures, with characteristic sublinks $K_1$ and $K_2$, respectively.  Notice that there is a natural symmetry of $S^{3}_{1,1}(L_n)$ exchanging the $K_i$, therefore also exchanging the spin structures. Thus, by modifying the boundary identification, we can use any spin 4-manifold $X_i'$ with boundary $S^{3}_{1,1}(L_n)$ to fill either boundary spin structure. Let us now use $X'$ for one of these spin 4-manifolds. 
To construct an $X'$ we perform the Kirby moves in Figure \ref{fig:torus} to the characteristic sublink $K_1$ in order to transform $K_1$ into the empty link. The first move is the move described in Figure \ref{fig:move}.
Notice that this move requires $n$ odd. The second move is an isotopy (and dropping the red curve from the picture, as it is not in $L'$).  The third is $-1+n^2-1$ positive blow ups, and the final move is one negative blow down. 

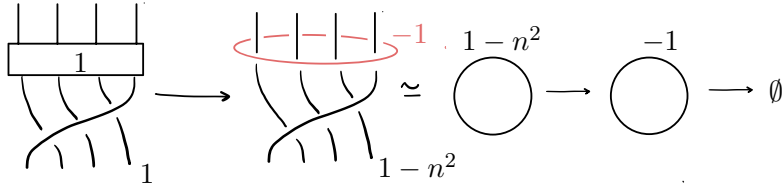
\begin{figure}[H]
\centering
\begin{overpic}[width=.6\textwidth, tics=20]
  {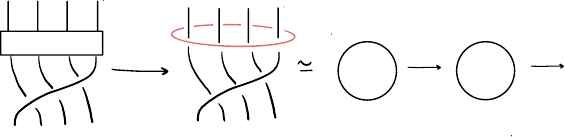}
   \put(50, 0){$1$}
      \put(140, 2){$1-n^2$}
            \put(172, 50){$1-n^2$}
                        \put(240, 50){$-1$}
                                       \put(288, 31){$\emptyset$}
   \put(145, 53){$\color{lred}{-1}$}
      \put(25, 42){$1$}
  \end{overpic}
\caption{Moves in the characteristic link calculus which remove a $1$ framed $T(n,n+1)$ torus knot, for $n$ odd.}
\label{fig:torus}
\end{figure}

The manifold we have obtained from performing these modifications to the $(1,1)$-trace on $L_n$ is our spin filling $X'$. From Figure \ref{fig:torus}, we compute $b_2(X')$ and $\sigma(X')$.
$$b_2(X')=2+1+n^2- 1-1-1=n^2.$$
$$\sigma(X')=1-1+n^2- 1-1+1=n^2-1.$$


If $S^3_{1,1}(L_n)$ were 0-surgery on some knot $K$, then we could glue $-X_0(K)$ and $X'$ along $S^3_{1,1}(L_n)$ (using any gluing map which appropriately pairs the boundary spin structures) to get a closed simply connected spin 4-manifold $Y$.  Since $b_2(X_0(K))=1$ and $\sigma(X_0(K))=0$ we have that
$$b_2(Y)=b_2(X')+b_2(X_0(K_1))=n^2+1, \quad \sigma(Y)=\sigma(X')+\sigma(X_0(K_1))=n^2-1.$$
Applying Theorem \ref{thm:10/8} to the closed spin 4-manifold $Y$, we have
$$4(n^2+1)\geq 5(n^2-1)+8$$
which is a contradiction for all integers $n\geq 2$.  
Thus we conclude that $S^3_{1,1}(L_n)$ cannot be surgery on a knot for any (odd) integer $n\geq 2$.  
\qed
%



\vspace{3mm}

\noindent
\textbf{Proof of Theorem \ref{thm:evenknot}:} Suppose for a contradiction that for any fixed odd integers $p$ and $q$, $S^{3}_{p,q}(L_n)$ is a surgery on a knot $K\subset S^3$. More specifically assume that $S^3_{p,q}(L_n)=S^3_{s/t}(K)$, for some knot $K$, some choice of $s=\pm(pq-1)$, and some natural number $t$ prime to $s$. 
We already proved the case when $s=0$. We now assume $s$ is nonzero.  It is well known that
$$S^{3}_{s/t}(K)\# L(t, s)=S^{3}_{st}(K_{t,s})$$ 
where $K_{t,s}$ denotes the $(t,s)$-cable on $K$, and $L(t,s)$ is the lens space obtained as $S^{3}_{t/s}(U)$ for the unknot $U$, see for example \cite[Proposition 4]{Moser}.  We will use this identification to argue similarly as we did in the proof of  Corollary \ref{coro:even}; here is an outline.

Observe that $S^3_{p,q}(L_n)=S^3_{s/t}(K)$ has two spin structures $\mathfrak{s'_p}$ and $\mathfrak{s'_q}$, which in the characteristic sublink calculus correspond to $K_1$ and $K_2$ respectively, and observe that $L(t, s)$ has a unique spin structure.
For $i\in\{p,q\}$, we will build a spin filling $X_{n,i,s,t}$ of $(S^{3}_{p,q}(L_n)\#L(t,s),\mathfrak{s}_i)$, where $\mathfrak{s}_i$ denote the two spin structures on $S^{3}_{p,q}(L_n)\#L(t,s)$. The Betti number and signature of this filling will have the property that for a fixed $p,q,$ and $i\in\{p,q\}$ and any choice of $s$, $t$ and $K$,
\begin{equation}\label{eq:goodlimit}
\lim_{n\to \infty}\frac{|\sigma(X_{n,i,s,t})|}{b_2(X_{n,i,s,t})}=1.
\end{equation}
We will then glue one of the $X_{n,i,s,t}$ to the trace $-X_{st}(K)$ to obtain a closed spin manifold which violates Theorem \ref{thm:10/8}.

We now discuss how to build a spin filling $X_{n,i,s,t}$ of $(S^{3}_{p,q}(L_n)\#L(t,s),\mathfrak{s}_i)$. By Corollary \ref{coro:lens}, we know that  $L(t, s)$ has a spin filling $V_{t,s}$ with $b_2$ and $|\sigma|$ both bounded above by $|s|$ (so in particular, where both $b_2$ and $\sigma$ are bounded in terms of the $p$ and $q$ we fixed at the outset, and the bounds are independent of the surgery coefficient $\frac{s}{t}$). By arguing as we did in Figure \ref{fig:torus}, for $i\in\{p,q\}$ we obtain spin fillings $X'_{n,i}$ of $(S^3_{p,q}(L_n),\mathfrak{s_i})$ 
which have  $$\lim_{n\to\infty}\frac{|\sigma(X'_{n,i})|}{b_2(X'_{n,i})}=1.$$
Thus, for a fixed $p$ and $q$, and $i\in \{p,q\}$, it is the case that for any knot $K$ and surgery coefficient $\frac{s}{t}$, the boundary connected sum $X_{n,i,s,t}\colon=V_{t,s}\natural X'_{n,i}$ is a spin filling of $(S^3_{p, q}(L_n)\# L(t,s),\mathfrak{s}_i)$ which satisfies equation \ref{eq:goodlimit}.

Since by assumption $ts$ is even, the trace $X_{st}(K_{t,s})$ of $K_{t,s}$ with framing $st$ is a spin filling of $S^{3}_{st}(K_{t,s})$ for one spin structure $\mathfrak{s_i}$ on $S^{3}_{st}(K_{t,s})$, without loss of generality let's assume $i=p$. 
Then we can glue $X_{n,p,s,t}$ to $-X_{st}(K_{t,s})$ to get a closed spin manifold, which can now readily be seen to violate Theorem \ref{thm:10/8} when $n$ is sufficiently large. 
\qed 

\bibliography{bibliography.bib} 

\begin{thebibliography}{10}

\bibitem{Auck}
David Auckly.
\newblock Surgery numbers of {$3$}-manifolds: a hyperbolic example.
\newblock In {\em Geometric topology ({A}thens, {GA}, 1993)}, volume 2.1 of
  {\em AMS/IP Stud. Adv. Math.}, pages 21--34. Amer. Math. Soc., Providence,
  RI, 1997.

\bibitem{BL}
Steven Boyer and Daniel Lines.
\newblock Surgery formulae for {C}asson's invariant and extensions to homology
  lens spaces.
\newblock {\em J. Reine Angew. Math.}, 405:181--220, 1990.

\bibitem{Dvafa}
Andrew Donald and Faramarz Vafaee.
\newblock A slicing obstruction from the {$\frac{10}{8}$} theorem.
\newblock {\em Proc. Amer. Math. Soc.}, 144(12):5397--5405, 2016.

\bibitem{Furuta}
M.~Furuta.
\newblock Monopole equation and the {$\frac{11}8$}-conjecture.
\newblock {\em Math. Res. Lett.}, 8(3):279--291, 2001.

\bibitem{GomS}
Robert~E. Gompf and Andr\'{a}s~I. Stipsicz.
\newblock {\em {$4$}-manifolds and {K}irby calculus}, volume~20 of {\em
  Graduate Studies in Mathematics}.
\newblock American Mathematical Society, Providence, RI, 1999.

\bibitem{HKMP}
Matthew Hedden, Min~Hoon Kim, Thomas~E. Mark, and Kyungbae Park.
\newblock Irreducible 3-manifolds that cannot be obtained by 0-surgery on a
  knot.
\newblock {\em Trans. Amer. Math. Soc.}, 372(11):7619--7638, 2019.

\bibitem{HKL}
Jennifer Hom, \c{C}a\u{g}r\i Karakurt, and Tye Lidman.
\newblock Surgery obstructions and {H}eegaard {F}loer homology.
\newblock {\em Geom. Topol.}, 20(4):2219--2251, 2016.

\bibitem{HL}
Jennifer Hom and Tye Lidman.
\newblock A note on surgery obstructions and hyperbolic integer homology
  spheres.
\newblock {\em Proc. Amer. Math. Soc.}, 146(3):1363--1365, 2018.

\bibitem{Kaplan}
Steve~J. Kaplan.
\newblock Constructing framed {$4$}-manifolds with given almost framed
  boundaries.
\newblock {\em Trans. Amer. Math. Soc.}, 254:237--263, 1979.

\bibitem{Lic}
W.~B.~R. Lickorish.
\newblock A representation of orientable combinatorial {$3$}-manifolds.
\newblock {\em Ann. of Math. (2)}, 76:531--540, 1962.

\bibitem{Moser}
Louise Moser.
\newblock {Elementary surgery along a torus knot.}
\newblock {\em Pacific Journal of Mathematics}, 38(3):737 -- 745, 1971.

\bibitem{sivekzentner}
Steven Sivek and Raphael Zentner.
\newblock Surgery obstructions and character varieties.
\newblock {\em Trans. Amer. Math. Soc.}, 375(5):3351--3380, 2022.

\bibitem{Taube}
Clifford~Henry Taubes.
\newblock Gauge theory on asymptotically periodic {$4$}-manifolds.
\newblock {\em J. Differential Geom.}, 25(3):363--430, 1987.

\bibitem{Wall}
Andrew~H. Wallace.
\newblock Modifications and cobounding manifolds.
\newblock {\em Canadian J. Math.}, 12:503--528, 1960.

\end{thebibliography}
\bibliographystyle{plain}

\Addresses

\end{document}